\documentclass[10pt]{article}

\usepackage{times}
\usepackage{amsthm}
\usepackage{amsmath}
\usepackage{amssymb}
\usepackage{mathrsfs}
\usepackage{pb-diagram}
\usepackage{xcolor}
\usepackage{tikz}
\usepackage{tikz-cd}
\usepackage{comment}

\usepackage[colorlinks]{hyperref}
\def\o{\omega}
\def\si{\sigma}
\def\Si{\Sigma}

\def\th{\theta}
\def\Th{\Theta}
\def\O{\Omega}
\def\N{\mathbb{N}}
\def\T{\mathbb{T}}

\def\X{{\sf X}}

\def\D{{\mathcal D}}

\def\J{\mathrm J}

\def\R{\mathbb{R}}

\def\S{\mathcal S}

\def\m{{\sf m}}
\def\wm{\widehat{\sf m}}

\def\({\left(}
\def\[{\left[}
\def\){\right)}
\def\]{\right]}

\def\bu{{\bullet}}

\def\p{\parallel}
\def\<{\langle}
\def\>{\rangle}

\newtheorem{Theorem}{Theorem}[section]
\newtheorem{Remark}[Theorem]{Remark}
\newtheorem{Lemma}[Theorem]{Lemma}

\newtheorem{Corollary}[Theorem]{Corollary}

\newtheorem{Definition}[Theorem]{Definition}
\newtheorem{Example}[Theorem]{Example}

\numberwithin{equation}{section}

\begin{document}

%-------------------------------------------------------------------------------------------------------
% Title
%-------------------------------------------------------------------------------------------------------

\title{Gohberg Lemma and Spectral Results for\\ Pseudodifferential Operators\\ on Locally Compact Abelian Groups}

\date{\today}

\author{N\'estor Jara and Marius M\u antoiu \textdagger
\footnote{
\newline
{%M. M. has been supported by the Fondecyt Project 1160359. 
}}
}

\date{\small}

\maketitle \vspace{-1cm}

%-------------------------------------------------------------------------------------------------------
% Abstract
%-------------------------------------------------------------------------------------------------------

\begin{abstract}
We provide a new type of proof for known or new Gohberg lemmas for pseudodifferential operators on Abelian locally compact groups $\X$\,. We use $C^*$-algebraic techniques, which also give spectral results to which the Gohberg lemma is just a corollary. These results extend most of those appearing in the literature in various directions. In particular, compactness or a a Lie structure are not needed. The ideal of all the compact operators in ${\rm L}^2(\X)$ is replaced by all the ideals having a crossed product structure, which is a consistent generalization. We also indicate several new examples, mostly connected to specific behaviors of functions on the dual $\Xi$ of $\X$\,.  
\end{abstract}

%\tableofcontents

%---------------------------------------------------------------------------------------------------------
\section{Introduction}\label{didirlish}
%-----------------------------------------------------------------------------------------------------------

Roughly, a Gohberg-type lemma (called {\it $G$-lemma} from now on) is a lower bound for the distance of an operator ${\sf Op}(f)$\,, defined by a symbol $f$ and acting in a Hilbert space $\mathscr H$, to the ideal $\mathbf K(\mathscr H)$ of compact operators. The theory started long ago \cite{Go,Gr,Ho,KN,Se}, but recently it resurfaced as a consequence of certain progress in Harmonic Analysis \cite{DR,Man,MW,RVR,VR,VR1}. The new background is that of global pseudodifferential operators for groups that might be different from $\R^n$, as treated in \cite{RT1,RT2} (see also the references therein). At the beginning this concerned singular integral operators, but soon it was extended to pseudodifferential operators. The results mostly involved groups $\X$\,, with $\mathscr H={\rm L}^2(\X)$\,, \cite{RVR} being a notable exception. Part of the interest of $G$-lemmas is the fact that the mentioned lower bound, although referring to a highly non-commutative situation, is expressed directly on values taken by the symbol, more precisely on its asymptotic behavior at infinity in a certain sense.

\smallskip
This article contributes to a different and hopefully interesting point of view upon $G$-lemmas. The approach is through $C^*$-algebras and the basic tool is the simple observation that spectra of operators or of elements in $C^*$-algebras are invariant by applying monomorphisms. Analytical arguments are not needed. The $G$-lemma will then be a consequence and in fact one will get an equality instead of a lower bound, which is an improvement. 

\smallskip
Reaching a situation in which such a principle can be applied needs studying the connection between pseudodifferential operators on Abelian locally compact groups (in our case) and $C^*$-algebras. Those appearing here are {\it crossed products} involving {\it the action} $\Th$ of the dual group $\Xi:=\widehat\X$ on Abelian algebras of functions defined on $\Xi$\,. These are basically composed of complex functions on $\Xi\!\times\!\Xi$\,, but a partial Fourier transformation in the first variable converts them in functions $f:\X\!\times\!\Xi\to\mathbb C$\,; we show that the pseudodifferential calculus ${\sf Op}$ can be applied on such ``symbols". A priori, this does not make use of H\"ormander-type conditions, which are actually not defined in many cases. When they are available, as a rule, the conditions emerging from the algebraic framework are more general and do not make use of an infinity of conditions on derivatives. See however Remark \ref{frustiuk}, making the comparison with Grushin's treatment of the case $\X=\R^n$.

\smallskip
The formalism allows to deal with a large class of algebras, coding the behavior of the symbol at infinity, as well as ideals which may be different from $\mathbf K\big[{\rm L}^2(\X)\big]$ (this one will be labeled as {\it the standard case}). One obstacle in finding the right arguments is the fact that the mentioned isomorphisms only takes place in suitable quotients of $C^*$-algebras, which are not naturally represented in ${\rm L}^2(\X)$ or in some related Hilbert space.

\smallskip
We sketch briefly the content of the paper.

\smallskip
A description of the framework is provided in Section \ref{brocart}.

\smallskip
In Section \ref{bulgaria} we start with a (non-unital) $C^*$-subalgebra $\J$ of bounded uniformly continuous functions on $\Xi$ which is invariant under translations and contains all the functions that decay at infinity. One calls such $C^*$-subalgebra {\it admissible}. To it one canonically attaches a unital Abelian $C^*$-algebra $\widetilde\J$\,, also invariant under translations, by the condition $\Th_\zeta(\psi)-\psi\in\J$\,, $\forall\,\zeta\in\Xi$\,. One obtains that $\J$ is now an ideal of $\widetilde{\J}$\,. Examples in Section \ref{adunate} show that, typically, $\widetilde\J$ is much larger than $\J$\,. Elements of $\widetilde\J$ also show up as continuous function on $\widetilde\Si$\,, its Gelfand spectrum, a compactification of $\Xi$ on which the action of $\Xi$ on itself by translations extends continuously. The Gelfand spectrum of the ideal $\J$ may be seen as an invariant open subset of $\widetilde\Si$\,, containing $\Xi$\,. The couple $\big(\,\widetilde\J,\widetilde\Si\,\big)$ models the allowed behavior of the symbol $\X\!\times\!\Xi\ni(x,\xi)\to f(x,\xi)$ in the variable $\xi$\,, while $(\J,\Si)$ defines the behavior that we want to ``factor out". An important role is played by the quotient $C^*$-algebra $\J^\bu\!:=\widetilde \J/\J$, its Gelfand spectrum $\Si^\bu\!=\widetilde\Si\!\setminus\!\Si$ and the topology of $\widetilde\Si$ around $\Si^\bu$.

\smallskip
The presence of the couple of variables $(x,\xi)$\,, belonging to the ``phase space" $\X\!\times\!\Xi$\,, is at the root of the non-commutativity and the sophistication of the pseudodifferential calculus. In Section \ref{frogart} we show how one ``dilates" the constructions from the preceding section to this setting. As a first step, every invariant $C^*$-algebra $\mathrm A$ of bounded continuous functions on $\Xi$ generates a non-Abelian crossed product $C^*$-algebra $\Xi\!\rtimes\!\mathrm A$\,. Its product mixes in a convenient way the convolution in the group $\Xi$ with the action $\Th$ of $\Xi$ on $\mathrm A$\,. A dense $^*$-subalgebra of $\Xi\!\rtimes\!\mathrm A$ is ${\rm L}^1(\Xi;\mathrm A)$\,. This well-known construction \cite{Wi} will be applied both to $\mathrm A=\J$ and to $\mathrm A=\widetilde\J$\,. A partial Fourier transformation will then provide a couple formed of a $C^*$-algebra $\widetilde{\mathfrak B}$ and an ideal $\mathfrak B$\,. Remarkably, the pseudodifferential prescription maps isomorphically $\widetilde{\mathfrak B}$ into the operator algebra $\widetilde{\mathbf B}\!:={\sf Op}\big(\widetilde{\mathfrak B}\big)$\,, containing the ideal $\mathbf B\!:={\sf Op}(\mathfrak B)$\,. The later is also an ideal in the entire $\mathbf B\big[{\rm L}^2(\X)\big]$ precisely when $\J=\mathrm C_0(\Xi)$\,, implying $\mathbf B=\mathbf K\big[{\rm L}^2(\X)\big]$\,. The important quotient $\mathbf B^\bu\!:=\widetilde{\mathbf B}/\mathbf B$ can be viewed as a generalization of the Calkin algebra.

\smallskip
In Section \ref{froggy} we prove our main results, Theorem \ref{dagorlat} (an identification of the spectrum of the image of ${\sf Op}(f)$ in $\mathbf B^\bu$) and Corollary \ref{gondolin} (a $G$-lemma). For suitable functions $f:\X\!\times\!\Xi\to\mathbb C$ we define the operator in $\mathscr H:={\rm L}^2(\X)$ 
\begin{equation}\label{zurzurel}
\begin{aligned}
\big[{\sf Op}(f)u\big](x):&=\int_\X\int_\Xi\xi\big(xy^{-1}\big)f(x,\xi)u(y)dyd\xi\\
&=\int_\Xi\xi(x)f(x,\xi)\widehat u(\xi)d\xi
\end{aligned}
\end{equation}
(actually it can be defined in various situations, with various interpretations; we shall not go into details).
It is known that in most cases the values of the symbol are difficult to be converted sharply in properties of the operator. Under our conditions, however, the spectrum of the image of ${\sf Op}(f)$ in the quotient $\mathbf B^\bu$ is equal to the set of values taken by $f$  ``at infinity", i.e. on the set $\X_\infty\!\times\!\Si^\bu$, where $\X_\infty$ is the one-point  compactification of $\X$\,.
Such a result is possible since the action $\Th$ extended to $\widetilde\Si$ and then restricted to $\Si^\bu$ is trivial, being composed only of fixed points. This allows one of the critical steps in the sequence of isomorphisms from Step 2 of the proof. In Theorem \ref{dagorlat}, besides the mentioned statement, one also exhibit two other equivalent forms of the spectrum, involving neighborhoods of $\Si^\bu$ in the compact space $\widetilde\Si$ and their traces on the group $\Xi$\,, bringing the information at the level of values taken by the symbol on the  ``finite part" $\X\!\times\!\Xi$\,. Corollary \ref{gondolin} is our $G$-lemma, deduced from the Theorem by passing through the spectral radius; the $C^*$-algebras appearing in \eqref{sofia} are Abelian, so their elements are normal. This is the reason for having an equality in our $G$-lemma  and not only a lower bound.

\smallskip
We dedicate Section \ref{stindarde} to an important situation, that we call {\it the standard case}. The starting admissible $C^*$-subalgebra is the smallest one $\J_{\rm st}\equiv\textrm C_0(\Xi)$\,, giving rise at the non-commutative level to the ideal of compact operators. In this situation, the compactification $\widetilde\Si$ is just the disjoint union between $\Xi$ and $\Si^\bu$, which allows a more concrete presentation of the results, making use of the natural asymptotic behavior of the symbol outside large compact sets. There is a price to pay, equalities being now replaced by set inclusions or lower bounds in the $G$-lemma. The result is compared to a statement in \cite{Gr},

\smallskip
The examples of Section \ref{adunate} have various degrees of generality. We present concrete examples of admissible $C^*$-subalgebras $\J$, which are also ideals of the $C^*$-algebra of all uniformly continuous bounded functions. Elements of the larger algebra $\widetilde\J$ also appear, sometimes in an explicit form. 

\smallskip
The reference \cite{Man} has some similitudes with the present article. $C^*$-algebras appear in the main statements, as well as general ideals. However, the analytical  proofs are much closer to the classical arguments used to prove $G$-lemmas in other situations. As a consequence, \cite{Man} is only able to treat {\it compact} Abelian groups (those for which $\Xi$ is discrete). This is an important limitation. In addition, a spectral result is missing and the $G$-lemma only consists of a lower bound. 

\smallskip
An expert in topology could notice that filters might have been used systematically all over the presentation. We preferred not to do so, but they appear occasionally in a proof in Section \ref{stindarde}.

\smallskip
The commutativity of the group is a serious restriction. In \cite{DR} non-commutative groups appear. They have to be compact Lie groups. The sophistication of the pseudodifferential calculus in this case \cite{RT1} makes its way in the statement of the $G$-lemma. It would be interesting to find a treatment for locally compact groups which are neither Lie and compact, nor Abelian.

%---------------------------------------------------------------------------------------------------------
\section{Framework and conventions}\label{brocart}
%-----------------------------------------------------------------------------------------------------------

We start with some notations and conventions that will be used all over the paper. If $\mathscr H$ is a complex Hilbert space, $\mathbf B(\mathscr H)$ will denote the $C^*$-algebra of all bounded linear operators acting in $\mathscr H$. The ideal of compact operators will be denoted by $\mathbf K(\mathscr H)$\,. Our locally compact spaces will always be Hausdorff. For such a space $Z$ we set ${\rm BC}(Z)$ for the space of all bounded and continuous complex functions defined on $Z$. Those which are uniformly continuous (whenever this makes sense) are elements of ${\rm BC_u}(Z)$ and those which tend to zero at infinity (outside arbitrarily large compact sets) are elements of  ${\rm C}_0(Z)$\,. One also encounters ${\rm C}_0(Z,\mathfrak Y)$\,, where $\big(\mathfrak Y,\star,^\star,\p\!\cdot\!\p_\mathfrak Y\!\big)$ is a $C^*$-algebra, with norm
$$
\p\!f\!\p_{{\rm C}_0(Z,\mathfrak Y)}\,:=\sup_{z\in Z}\p\!f(z)\!\p_\mathfrak Y\,,
$$
involution $f^*(\cdot):=f(\cdot)^\star$ and product $(fg)(\cdot):=f(\cdot)\star g(\cdot)$\,.

\smallskip
By {\it ideal} in a $C^*$-algebra we mean a two-sided closed ideal which is stable under the involution. {\it Morphisms} between $C^*$-algebras (in particular Hilbert space representations) are assumed by definition to intertwine the involutions.  For two $C^*$-algebras $\mathfrak A_1$ and $\mathfrak A_2$ we write $\mathfrak A_1\!\ll\mathfrak A_2$ if  $\mathfrak A_1$ is a $C^*$-subalgebra of $\mathfrak A_2$ and $\mathfrak A_1\!\lll\mathfrak A_2$ if $\mathfrak A_1$ is an ideal of $\mathfrak A_2$\ (with the convention above). With previous notations one has
\begin{equation*}\label{trol}
\mathbf K(\mathscr H)\lll\mathbf B(\mathscr H)\,,\quad{\rm BC_u}(Z)\ll{\rm BC}(Z)\,,\quad{\rm C}_0(Z)\lll{\rm BC}(Z)\,.
\end{equation*}

Let $\X$ be a locally compact Abelian group, with (Abelian) Pontryagin dual $\widehat\X\equiv \Xi$. Both group laws will be denoted multiplicatively. To make things interesting, we will assume that $\Xi$ is {\it not compact}, i.\,e. $\X$ is not discrete. However, note that $\X$ is allowed to be compact. 

\smallskip
On $\X$ and $\Xi$ we consider correlated Haar measures $d\m(x)\equiv dx$ and $d\wm(\xi)\equiv d\xi$\,. The Fourier transformation $\mathcal F:{\rm L}^1(\X)\to{\rm C}_0(\Xi)$ can also be defined as a unitary operator $\mathcal F:{\rm L}^2(\X)\to{\rm L}^2(\Xi)$ and it is essentially given by 
\begin{equation*}\label{glaurung}
(\mathcal F\varphi)(\xi):=\int_\X\overline{\xi(x)}\varphi(x)dx\,,
\end{equation*}
with inverse
\begin{equation*}\label{morgoth}
\big(\mathcal F^{-1}\psi\big)(x):=\int_{\Xi}\xi(x)\psi(\xi)d\xi\,.
\end{equation*}

The Riemann-Lebesgue Lemma and the fact that $\X$ is isomorphic to the dual of $\Xi$ (by the Pontryagin-van Kampen Theorem) also insures the interpretation $\mathcal F^{-1}\!:{\rm L}^1(\Xi)\to{\rm C}_0(\X)$ (linear contraction). Actually it extends to an isomorphism between ${\sf C}^*(\Xi)$\,, the group $C^*$-algebra of $\Xi$\,, and ${\rm C}_0(\X)$\,.

\smallskip
We denote by $\th:\Xi\to{\rm Homeo}(\beta\Xi)$ the canonical  extension to $\beta\Xi$\,, the Stone-Ce\u ch compactification of $\Xi$\,,  of the action of $\Xi$ on itself by translations:
\begin{equation*}\label{dwarf}
\th_\eta(\xi):=\xi\eta\,,\quad\forall\,\xi,\eta\in\Xi\,,
\end{equation*} 
which lifts to an action on ${\rm BC}(\Xi)$ by 
\begin{equation*}\label{melkor}
\big[\Th_\eta(\psi)\big](\xi):=\psi\big[\th_{\eta}(\xi)\big]=\psi(\xi\eta)\,,\quad\forall\,\xi,\eta\in\Xi\,.
\end{equation*}

This one is pointwise continuous in $\eta$ only on uniformly continuous functions.

%---------------------------------------------------------------------------------------------------------
\section{Abelian ideals and $C^*$-algebras}\label{bulgaria}
%-----------------------------------------------------------------------------------------------------------

In this section we define suitable pairs of $C^*$-algebras of bounded uniformly continuous functions on $\Xi$\,, one of them being an ideal of the other. They are invariant under translations and their quotient has interesting dynamical properties.

\begin{Definition}\label{macedonia}
The $C^*$-subalgebra $\mathrm J(\Xi)$ of $\,\mathrm{BC_u}(\Xi)$ is called {\rm admissible} if (a) it contains $\mathrm C_0(\Xi)$ and (b) it is invariant under translations: if $\psi\in\mathrm J(\Xi)$ and $\eta\in\Xi$ then $\Th_\eta(\psi)\in\mathrm J(\Xi)$\,.
\end{Definition}

\begin{Definition}\label{serbia}
An admissible $C^*$-subalgebra $\mathrm J(\Xi)$ being given, we define 
\begin{equation*}\label{muntenegro}
\widetilde{\mathrm J}(\Xi):=\big\{\psi\in\mathrm{BC_u}(\Xi)\,\big\vert\,\Th_\zeta(\psi)-\psi\in\mathrm J(\Xi)\,,\ \forall\,\zeta\in\Xi\big\}\,.
\end{equation*}
\end{Definition}

When $\Xi$ is understood or irrelevant, we just write $\J$ and $\widetilde{\J}$ respectively. It is easy to show that $\widetilde{\mathrm J}$ is a unital $C^*$-algebra of uniformly continuous functions. It is also invariant under translations: for every $\psi\in\widetilde{\mathrm J}$ and every $\eta\in\Xi$ one has $\Theta_\eta(\psi)\in\widetilde{\mathrm J}$\,. We have
\begin{equation*}\label{skopjie}
\J_{\rm st}:=\mathrm C_0\lll\J\lll\widetilde\J\gg\widetilde\J_{\rm st}\ggg\mathrm C_0\,,
\end{equation*}
so both $\mathrm C_0$ and $\J$ are ideals in $\widetilde\J$\,. The label ${\rm st}$ refers to what will be called subsequently {\it the standard case}, since it is the $C^*$-subalgebra (and ideal) usually considered for Gohberg lemmas \cite{Gr}. Even in non-standard cases, as in \cite{Man}, the spaces of symbols are usually ideals of $\mathrm{BC_u}$, but for our construction a $C^*$-subalgebra is enough. Note, however, that if $\J$ is unital, then $\widetilde{\J}=\J$ and the construction becomes trivial. Because of this, we restrict to the case where $\J$ is non-unital, thus a proper ideal of $\widetilde{\J}$. A particular case appears when $\J$ is a (proper) ideal of $\mathrm{BC_u}$\,. 
 
\smallskip
The Gelfand spectrum $\widetilde{\Si}$ of $\widetilde\J$ is a regular compactification of $\Xi$ (i.\,e.\! we can see $\Xi$ as an open dense subset of $\widetilde{\Si}\,$)\,.  A neighborhood base of $\sigma\in \widetilde{\Si}$  is given by
\begin{equation}\label{surub}
\left\{W(\sigma|\epsilon;\varphi_1,\dots,\varphi_m)\, \,\big\vert\, \,\epsilon>0\,,\,\varphi_1,\dots,\varphi_m\in \widetilde{\J}\,,\,m\in \mathbb{N}\right\},
\end{equation}
where
\begin{equation}\label{capitan}
W(\sigma|\epsilon;\varphi_1,\dots,\varphi_m):=\left\{\nu\in \widetilde{\Si}\,\big| \,\,|\nu(\varphi_j)-\sigma(\varphi_j)|<\epsilon\,, \,\forall\,j=1,\dots,m\right\}.
\end{equation}

 Since $\widetilde\J$ is invariant, the action of $\Xi$ on itself by left translations extends to a continuous action on $\widetilde{\Si}$\,,  also denoted by $\th$. Note that $\widetilde{\Si}$ is a quotient of $\beta \Xi$ and this action on $\widetilde\Si$ is compatible with the previous $\th$ defined on $\beta \Xi$ through the canonical surjection. The Gelfand isomorphism $\widetilde\J\cong\mathrm C\big(\widetilde{\Si}\big)$ consists here in associating to any function belonging to $\widetilde{\J}$ its unique continuous extension to $\widetilde{\Si}$\,, denoted by $\widetilde{\psi}$\,.

\smallskip
The locally compact space  $\Si$\,, the Gelfand spectrum of $\J$\,, will not be needed explicitly. Note, however, that by the relation ${\rm C}_0\lll\J$\,, one can identify $\Xi$ with one of its open subsets; we may thus write
\begin{equation*}\label{simurg}
\Xi\subset\Si\subset\widetilde{\Si}\,.
\end{equation*} 

Let us denote by $\Si^\bu\!:=\widetilde{\Si}\!\setminus\!\Si$ the corresponding {\it corona set}, which is compact since $\Si$ is open 
in $\widetilde{\Si}$\,. It is (homeomorphic to) the Gelfand spectrum of the quotient $C^*$-algebra $\J^\bu\!:=\widetilde{\J}/\J$\,. We recall that $\J^\bu$ can be seen as the set of characters of $\widetilde\J$ which are trivial on the ideal $\J$\,.  Extending terminology used for $\J=\mathrm C_0$ \cite{Ro}\,, we could call $\widetilde{\Si}$ {\it the $\J$-Higson compactification} and $\Si^\bu$ {\it the $\J$-Higson corona}. When $\J\ne\mathrm C_0$\,, there is an extra compact space $\widetilde{\Si}\!\setminus\!\Xi$ containing strictly ${\Si}^\bu$.

\smallskip
We are interested now in the dynamical properties of this quotient ${\J}^\bu\!=\widetilde\J/\J\cong\mathrm C\big({\Si}^\bu\big) $\,. 
If $\zeta\in\Xi$ and $\si\in{\Si}^\bu$, for every $\psi\in\widetilde\J$ one has
\begin{equation*}\label{tirana}
\big(\th_\zeta(\si)-\si\big)(\psi)=\si\big(\Th_\zeta(\psi)-\psi\big)=0\,,
\end{equation*}
since $\Th_\zeta(\psi)-\psi\in\J$\,. This shows that
\begin{equation*}\label{atena}
\th_\zeta(\si)=\si\,,\quad\forall\,\si\in\Si^\bu,\,\zeta\in\Xi\,.
\end{equation*}

This will be crucial for our arguments, so we make a formal statement:

\begin{Lemma}\label{tinguiririca}
The elements of $\,{\Si}^\bu$ are all fixed points under the action $\th$\,. 
\end{Lemma}

%---------------------------------------------------------------------------------------------------------
\section{Algebras of pseudifferential operators on Abelian compact groups}\label{frogart}
%-----------------------------------------------------------------------------------------------------------

Relying on the couple $\big(\,\widetilde \J,\J\big)$\,, we are going to introduce $C^*$-algebras of symbols and connect them to the ${\sf Op}$-calculus via Hilbert space representations. We will make use below, sometimes implicitly, of the known monomorphisms (each time, the first two are isomorphisms)
\begin{equation*}\label{laba}
{\rm C}_0(\X)\!\otimes\!\widetilde \J\cong {\rm C}_0(\X)\!\otimes\!{\rm C}\big(\,\widetilde{\Si}\,\big)\cong {\rm C}_0\big(\X\!\times\!\widetilde{\Si}\,\big)\,\hookrightarrow {\rm C}(\X_\infty\!\times\! \widetilde{\Si})\,,
\end{equation*}
\begin{equation}\label{labuta}
{\rm C}_0(\X)\!\otimes\!\J^\bu\cong {\rm C}_0(\X)\!\otimes\!{\rm C}\big(\Si^\bu\big)\cong {\rm C}_0\big(\X\!\times\!\Si^\bu\big)\,\hookrightarrow {\rm C}\big(\X_\infty\!\times\!{\Si^\bu}\big)\,,
\end{equation}
where $\X_\infty$ denotes the Alexandrov compactification of $\X$ if it is not compact and $\X$ if it is.

\smallskip
With the action $\Th$ one can form the crossed product $\widetilde{\mathfrak J}:=\Xi\!\ltimes\!\widetilde \J$\,.  It is the enveloping $C^*$-algebra of the Banach $^*$-algebra ${\rm L}^1\big(\Xi;\widetilde \J\,\big)$ (isomorphic as a Banach space to the projective tensor product ${\rm L}^1\big(\Xi)\overline\otimes \,\widetilde \J$\,\big)\,, with the obvious norm, the composition law
\begin{equation*}\label{sauron}
(\Phi\diamond\Psi)(\xi):=\int_{\Xi}\Th_\eta\big[\Phi(\xi\eta^{-1})\big]\Psi(\eta)d\eta
\end{equation*}
and the involution
\begin{equation*}\label{mordor}
\Psi^\diamond(\xi):=\Th_\xi\big[\Psi(\xi^{-1})^*\big]\,.
\end{equation*}

Using notations as $\big[\Psi(\xi)\big](\zeta)\equiv\Psi(\xi,\zeta)$\,, the algebraic structure of ${\rm L}^1\big(\Xi;\widetilde \J\,\big)$ can be reformulated as 
\begin{equation*}\label{ungoliant}
(\Phi\diamond\Psi)(\xi,\zeta):=\int_{\Xi}\Phi\big(\xi\eta^{-1}\!,\zeta\eta\big)\Psi(\eta,\zeta)d\eta\,,
\end{equation*}
\begin{equation*}\label{angband}
\Psi^\diamond(\xi,\zeta):=\overline{\Psi\big(\xi^{-1}\!,\zeta\xi\big)}\,.
\end{equation*}

With the invariant ideal $\J$ one can make the same crossed product construction, leading to the $C^*$-algebra $\mathfrak J\!:=\Xi\!\ltimes\!\J$\,. In the following we will use the abbreviations $\widetilde{\mathfrak L}:={\rm L}^1\big(\Xi;\widetilde \J\,\big)$ and $\mathfrak L:={\rm L}^1\big(\Xi;\J\big)$. It is known that the obvious inclusion $\mathfrak L\subset\widetilde{\mathfrak L}$ extends to a monomorphism between the corresponding crossed products, the first one being an ideal of the second, thus finally $\mathfrak J\lll\widetilde{\mathfrak J}$\,.

\smallskip
There is a natural (Schr\"odinger) representation of $\widetilde{\mathfrak J}$ in the Hilbert space ${\rm L}^2(\Xi)$ given for $\Psi\in \widetilde{\mathfrak L}$ by
\begin{equation*}\label{haradrim}
\big[{\sf Sch}(\Psi)w\big](\xi):=\int_{\Xi}\Psi\big(\xi\eta^{-1}\!,\eta\big)w(\eta)d\eta\,, \hspace{.4cm}\forall\, w\in {\rm L}^2(\Xi)\,,\, \xi\in \Xi\,.
\end{equation*}

It is clear that ${\sf Sch}$ also transforms isomorphically the Hilbert space ${\rm L}^2(\Xi\!\times\!\Xi)$ into the Hilbert space $\mathbf B^2\big[{\rm L}^2(\Xi)\big]$ of all Hilbert-Schmidt operators in ${\rm L}^2(\Xi)$\,.

\smallskip
Most elements of $\widetilde{\mathfrak J}$ are functions (of a certain type)  defined on $\Xi\!\times\!\Xi$\,. But the symbols of pseudodifferential operators are naturally defined on $\X\!\times\!\Xi$\,. We will turn to this situation using a partial Fourier transformation in the first variable: for suitable $g:\X\!\times\!\Xi\to \mathbb{C}$\,, we consider $\mathbb F_{(1)} g:\Xi\!\times\!\Xi\to \mathbb{C}$ by
\begin{equation*}\label{paraleu}
\big(\mathbb F_{(1)} g\big)(\eta,\xi):=\int_\X\overline{\eta(x)}g(x,\xi)dx\,,
\end{equation*} 
which in some sense is the tensor product between $\mathcal F$ and ${\sf id}$\,. Our construction is described basically by
\begin{equation*}\label{katerini}
\mathfrak B:=\mathbb F_{(1)}^{-1}\,\mathfrak J\,,\quad\widetilde{\mathfrak B}:=\mathbb F_{(1)}^{-1}\,\widetilde{\mathfrak J}\,.
\end{equation*}

What actually happens is that $\mathbb F_{(1)}^{-1}$ maps linearly, contractively and densely (but not multiplicatively) $\widetilde{\mathfrak L}$ into ${\rm C}_0(\X)\!\otimes\!\widetilde \J\cong {\rm C}_0\big(\X\!\times\!\widetilde{\Si}\,\big)$ (an easy extension of the Riemann-Lebesgue Lemma for Abelian locally compact groups). It also extends to the full crossed product $\widetilde{\mathfrak J}$\,: On the ${\rm L}^1$-level, one pushes forward the crossed product algebraic structure to the dense image of $\mathbb F_{(1)}^{-1}$\,, getting a (non-commutative) Banach $^*$-algebra. This structure is not needed in an explicit form (the involution is simply complex conjugation while the product is already a certain type of symbol composition). Then $\widetilde{\mathfrak B}$ is the enveloping $C^*$-algebra and the inverse partial Fourier transformation extends to an isomorphism of $C^*$-algebras. The same arguments leads to $\mathfrak B$ and one gets $\mathfrak B\lll\widetilde{\mathfrak B}$\,. For subsequent use, we mention below part of the conclusion:
\begin{equation}\label{tronuri}
\mathbb F_{(1)}^{-1}\,\widetilde{\mathfrak L}\subset\widetilde{\mathfrak B}\cap {\rm C}_0\big(\X\!\times\!\widetilde{\Si}\,\big)\,.
\end{equation}

For suitable functions $f:\X\!\times\!\Xi\to\mathbb C$  we define the operator in $\mathscr H\!:={\rm L}^2(\X)$ (or in other function spaces, when the structure of $\X$ allows it)
\begin{equation*}\label{balrog}
\begin{aligned}
\big[{\sf Op}(f)u\big](x):&=\int_\X\int_{\Xi}\xi\big(xy^{-1}\big)f(x,\xi)u(y)dyd\xi\,.
\end{aligned}
\end{equation*}
It is an integral operator with kernel $\kappa_f(x,y):=\big(\mathbb F_{(2)}^{-1}f\big)\big(x,xy^{-1}\big)$\,, where $\mathbb F_{(2)}$ is the Fourier transformation in the second variable. Studying this kernel, it follows easily that ${\sf Op}$ is a Hilbert space isomorphism between ${\rm L}^2(\X\!\times\!\Xi)\cong {\rm L}^2(\X)\!\otimes\!{\rm L}^2(\Xi)$ and the space $\mathbf B^2(\mathcal{H})$ of all Hilbert-Schmidt operators in $\mathscr H$. 

\smallskip
We indicate now another $C^*$-algebraic meaning, making the connection with the above crossed products.
One makes use of the unitary equivalence
\begin{equation*}\label{busuioc}
\mathbf U_\mathcal F:\mathbf B[{\rm L}^2(\Xi)]\to\mathbf B(\mathscr{H})\,,\quad \mathbf U_\mathcal F(T):=\mathcal F^{-1} T\mathcal F.
\end{equation*}

The basic fact  is that {\it we can regard ${\sf Op}$ as a representation of the $C^*$-algebra} $\widetilde{\mathfrak B}=\mathbb F^{-1}_{(1)}\,\widetilde{\mathfrak J}$ (and other similar Fourier transformed crossed products) in the Hilbert space $\mathscr{H}$ and that the diagram
\begin{equation}\label{bigorneau}
\begin{diagram}
\node{\widetilde{\mathfrak J}}\arrow{e,t}{\mathbb F^{-1}_{(1)}} \arrow{s,l}{{\sf Sch}}\node{\widetilde{\mathfrak B}}\arrow{s,r}{{\sf Op}}\\ 
\node{\mathbf B[{\rm L}^2(\Xi)]} \arrow{e,t}{\mathbf U_\mathcal F} \node{\mathbf B(\mathscr{H})}
\end{diagram}
\end{equation}
commutes. The horizontal arrows are (isometric) isomorphisms and the vertical ones are injective representations of $C^*$-algebras.
The fact that the diagram commutes is a straightforward computation on $\widetilde{\mathfrak L}\subset\widetilde{\mathfrak J}$ followed by an obvious density argument. The final conclusion is that
\begin{equation*}\label{numenor}
{\sf Op}=\mathbf U_\mathcal F\circ{\sf Sch}\circ\mathbb F_{(1)}\,.
\end{equation*}

%---------------------------------------------------------------------------------------------------------
\section{A spectral result and a G-lemma}\label{froggy}
%-----------------------------------------------------------------------------------------------------------

To formulate our main results, one must introduce operator versions of the previously defined $C^*$-algebras. We set
\begin{equation}\label{mithrandir}
\mathbf{B}\!:={\sf Op}\big(\mathfrak B)\lll\widetilde{\mathbf{B}}:={\sf Op}\big(\widetilde{\mathfrak B}\big)\ll\mathbf B(\mathscr H)\,.
\end{equation}

The pair formed of the $C^*$-algebra $\widetilde{\mathbf{B}}$ and its ideal $\mathbf{B}$ will play a central role, as well as the quotient $\mathbf{B}^\bu\!:=\widetilde{\mathbf{B}}/\mathbf{B}$\,. We set ${\sf Op}^\bu(f)\!:={\sf Op}(f)+\mathbf{B}$ for the image of ${\sf Op}(f)$ in the quotient. One denotes by ${\rm sp}_\mathfrak A(T)$ the {\it spectrum} of an element $T$ in a $C^*$-algebra $\mathfrak A$ and by $\rho_\mathfrak A(T)$ its {\it spectral radius}. 

\smallskip
For $f\in \mathbb F_{(1)}^{-1}\,\widetilde{\mathfrak L}$\,, which is a function with domain $\X\!\times\! \Xi$\,, there is a unique continuous extension to $\X\!\times\! \widetilde{\Si}$\, (see also \eqref{tronuri}). We call this extension $\tilde{f}$ and write $f^\bu$ for its restriction to $\X\!\times\!{\Si}^\bu$.   Moreover, for such $f$, there is a unique continuous extension to $\X_\infty\!\times\! \Xi$ (see (\ref{labuta})), where $\X_\infty\!:=\X\cup\{x_\infty\}$ denotes the Alexandrov compactification of $\X$ if it is not compact and $\X_\infty\!:=\X$ otherwise. We call this extension $f_\infty$\,; it satisfies $f_\infty(x_\infty,\xi)=0$ for every $\xi\in\Xi$\,. Then one analogously defines $\widetilde{f}_\infty$ on $\X_\infty\!\times\!\widetilde{\Si}$ and $f^\bu_\infty$ on $\X_\infty\!\times\!\Si^\bu$. 

\begin{Theorem}\label{dagorlat}
Let $f\in\mathbb F_{(1)}^{-1}\,\widetilde{\mathfrak L}\subset \mathrm C_0\big(\X\!\times\!\widetilde{\Si}\big)$\,. Then 
\begin{equation}\label{gollum}
{\rm sp}_{{\mathbf B}^\bu}\big({\sf Op}^\bu(f)\big)=f^\bu_\infty\big(\X_\infty\!\times \!{\Si}^\bu\big)\,.
\end{equation}
\end{Theorem}

\begin{proof}
{\it Step 1.}  It follows from (\ref{mithrandir}) that the asignation 
$$
f+\mathfrak B\mapsto{\sf Op}^\bu(f)={\sf Op}(f)+\mathbf B
$$ 
provides the isomorphism ${\mathfrak B}^\bu\!=\widetilde{\mathfrak B}/\mathfrak B\cong\widetilde{\mathbf B}/\mathbf B ={\mathbf B}^\bu$\,. Thus
\begin{equation}\label{morfologic}
{\rm sp}_{{\mathbf B}^\bu}\!\big({{\sf Op}^\bu}(f)\big)={\rm sp}_{{\mathfrak B}^\bu}\!\big(f+\mathfrak B\big)\,,
\end{equation}
so one has to show that the two r.\,h.\,s. of \eqref{gollum} and \eqref{morfologic} agree.

\smallskip
{\it Step 2.} As a consequence of the above preparations and of some general facts that will be explained, one has the sequence of $C^*$-morphisms
\begin{equation}\label{sofia}
\begin{aligned}
\mathbf B^\bu\!\cong\widetilde{\mathfrak B}/\mathfrak B&\overset{(a)}{\cong}\widetilde{\mathfrak J}/\mathfrak J\overset{(b)}{\cong}\Xi\!\ltimes\!\big(\widetilde \J/\J\big)\overset{(c)}{\cong}\Xi\!\ltimes\!\mathrm C\big({\Si}^\bu\big)\\
&\overset{(d)}{\cong} {\sf C}^*(\Xi)\!\otimes \!{\rm C}\big({\Si}^\bu\big)\overset{(e)}{\cong}{\rm C}_0(\X)\!\otimes\! {\rm C}\big({\Si}^\bu\big)\\
&\overset{(f)}{\cong} {\rm C}_0\big(\X\!\times\!\Si^\bu\big)\overset{(g)}{\hookrightarrow}{\rm C}\big(\X_\infty\!\times\!\Si^\bu\big)\,,
\end{aligned}
\end{equation}
which is the core of the proof. At any step $(y)$ we set $\Gamma_y$ for the corresponding morphism; all except $\Gamma_g$ are in fact isomorphisms. Remarkably, the quotient $\mathbf B^\bu$ turns out to be Abelian. 

\smallskip
First, $\Gamma_a$ is obtained by passing to the quotient the isomorphism $\mathbb F_{(1)}:\widetilde{\mathfrak B}\to\widetilde{\mathfrak J}$\,, which also maps $\mathfrak B$ onto $\mathfrak J$\,. So 
\begin{equation*}\label{gamaa}
\Gamma_a\big(f+\mathfrak B\big):=\mathbb F_{(1)}f+\mathfrak J\,.
\end{equation*}

\smallskip
The isomorphism $\Gamma_b:\Xi\!\ltimes\!\widetilde{\mathrm J}/\Xi\!\ltimes\!\J\!\to\Xi\!\ltimes\!\big(\widetilde \J/\J\big)$ is a basic known fact in the theory of crossed products; these ones provide exact functors (see \cite[Proposition 3.19]{Wi}): The action $\Th$ on $\widetilde \J$ restricts to an action on the invariant ideal $\J$ and this defines an obvious action on the quotient. There are corresponding crossed products (in which we do not indicate the action at the level of notations). If $\Phi$ belongs to the dense subspace ${\rm C}_{\rm comp}\big(\Xi,\widetilde{\J}\,\big)$\,, one has
\begin{equation*}\label{sopron}
\big[\Gamma_b\big(\Phi+\mathfrak J\big)\big](\xi)=\Phi(\xi)+\J\,,\quad\forall\,\xi\in\Xi\,.
\end{equation*}

The isomorphism $\Gamma_c$ arises from the isomorphism $\mathscr{G}:\widetilde{\J}/\J\to{\rm C}(\Si^{\bu})\,,\ \mathscr{G}(\psi+\J)=\psi^\bu$, which is a matter of Gelfand theory. Then, this isomorphism lifts to an isomorphism on the crossed products (see \cite[Lemma 2.65]{Wi}). Explicitly 
$$
\big[\Gamma_c(\Psi)\big](\xi):=\mathscr{G}\big[\Psi(\xi)\big]\,.
$$

Thus, until now, for every $\xi\in\Xi$ and $\si\in{\Si}^\bu$ we get
\begin{equation}\label{gamacba}
\big[\Gamma_c\Gamma_b\Gamma_a(f+\mathfrak B)\big](\xi,\si)=\big[(\mathbb F_{(1)}f)(\xi)\big]^\bu(\si)\equiv\big(\mathbb F_{(1)}f\big)^\bu(\xi,\si)\,.
\end{equation}

\smallskip
The critical point is the fourth isomorphism: Based on the fact that the action of $\Xi$ on $\Si^\bu$ (and thus on ${\rm C}\big({\Si}^\bu\,\big)$\,) is trivial (see Lemma \ref{tinguiririca}), it follows by \cite[Lemma\,2.73]{Wi} that the corresponding crossed product is isomorphic to the tensor product of ${\rm C}\big({\Si}^\bu\,\big)$ with the group $C^*$-algebra ${\sf C}^*(\Xi)$\,. 

\smallskip
The fifth isomorphism follows immediately from the fact that ${\sf C}^*(\Xi)$ is isomorphic to ${\rm C}_0(\X)$ via an extension of the inverse Fourier transform $\mathcal F^{-1}\!:{\rm L}^1(\Xi)\to {\rm C}_0(\X)\equiv {\rm C}_0\big(\widehat\Xi\big)$\,, as indicated in \cite[Proposition\,3.1]{Wi}. This extends then to the tensor products.

\smallskip
The final morphisms $\Gamma_f$ and $\Gamma_g$ are a standard facts in the theory of Abelian $C^*$-algebras; see \eqref{labuta}. One gets
\begin{equation}\label{gamafed}
\big[\Gamma_g\Gamma_f\Gamma_e\Gamma_d(\Phi)\big](x,\si)=\big(\mathbb F_{(1)}^{-1}\Phi\big)^\bu_\infty\,(x,\si)\,,\quad\forall\,x\in\X_\infty\,,\,\si\in\Si^\bu.
\end{equation}

Let us set $\Gamma\!:=\Gamma_g\Gamma_f\Gamma_e\Gamma_d\Gamma_c\Gamma_b\Gamma_a$\,. Combining \eqref{gamacba} with \eqref{gamafed}, in which the extension-restriction indicated by $\,^\bu\,$\  only refers to the second variable, we obtain finally
\begin{equation*}\label{vormula}
 \left[\Gamma\!\left(f+\mathfrak B_{\mathfrak G}\right)\right](x,\sigma)=f^\bu_\infty(x,\sigma)\,,\hspace{.4cm}\forall\,(x,\sigma)\in \X_\infty\!\times\!{\Si}^\bu.  
\end{equation*}

\smallskip
{\it Step 3.} Note that the spectrum of $f^\bu_\infty$ in ${\rm C}\big(\X_\infty\!\times\!\Si^\bu\big)$ is simply its image, since it has a compact domain. Now, by applying Steps 1 and 2 and the obvious fact that spectra are invariant under injective $C^*$-morphisms, we finish the proof of the theorem.
\end{proof}

We are going now to rephrase \eqref{gollum} in terms of values taken by $f$ only on the phase space $\X\!\times\!\Xi$\,. For any base of neighborhoods $\mathcal B$ of $\Si^\bu$ in $\widetilde\Si$ we set $\mathcal B|_\Xi\!:=\{B\cap\Xi\!\mid\! B\in\mathcal B\}$ (by the density of $\Xi$ in $\widetilde\Si$\,, each such ``trace" $B\cap\Xi$ is non-void). Examples are $\mathcal B\equiv\mathcal V$, the family of all the neighborhoods, $\mathcal B\equiv\mathcal N$, the family of all the closed (thus compact) neighborhoods, or $\mathcal B\equiv\mathcal W$, the family of the neighborhoods of $\Si^\bu$ deduced from the collection \eqref{surub} ($W$ is a neighborhood of a set if and only if it is a union of neighborhoods of all of its points). 

\begin{Corollary}\label{dragorlat}
Let $f\in\mathbb F_{(1)}^{-1}\,\widetilde{\mathfrak L}\subset \mathrm C_0\big(\X\!\times\!\widetilde{\Si}\big)$ and let $\mathcal B$ be a base of neighborhoods of $\,\Si^\bu$ in $\widetilde\Si$\,. Then 
\begin{eqnarray}\label{golumm}
{\rm sp}_{{\mathbf B}^\bu}\big({\sf Op}^\bu(f)\big)&=&\bigcap_{A\in \mathcal B|_\Xi}\overline{f\big(\X\!\times\!A\big)}\,\nonumber\\
&=&\big\{\lambda\in \mathbb{C}\,\big\vert\, \forall\,\epsilon>0\,, A\in \mathcal B|_\Xi\,,\,\exists\,(x,\xi)\in \X\!\times \!A: |f(x,\xi)-\lambda|<\epsilon\big\}\,.
\end{eqnarray}
\end{Corollary}

\begin{proof}
The second equality follows by a direct use of basic definitions. By Theorem \ref{dagorlat}, we are left with showing that 
\begin{equation*}\label{shower}
f^\bu_\infty\big(\X_\infty\!\times \!{\Si}^\bu\big)=\bigcap_{A\in \mathcal B_\Xi}\overline{f\big(\X\!\times\!A\big)}\,.
\end{equation*}
It is also easy to see that if this is proven for one base $\mathcal B$\,, it will also hold for any other one. So we will work with $\mathcal B=\mathcal N$.

\smallskip
Let $N\in\mathcal N$; then $\X_\infty\!\times\!N$ is the closure of $\X\!\times\!(N\cap\Xi)$ in the compact space $\X_\infty\!\times\!\widetilde\Si$\,. The continuous function $\widetilde f_\infty:\X_\infty\!\times\!\widetilde\Si\to\mathbb C$ is also closed, so it commutes with closures. So we only need to show that
\begin{equation*}\label{jower}
f^\bu_\infty\big(\X_\infty\!\times \!{\Si}^\bu\big)=\bigcap_{N\in \mathcal N}\widetilde f_\infty\big(\X_\infty\!\times\!N\big)\,.
\end{equation*}

For each $(x,\sigma)\in \X_\infty\!\times\!{\Si}^\bu$ and $\lambda\neq f(x,\sigma)$\,, there is a compact neighborhood $V$ of $f(x,\sigma)$ with $\lambda\not\in V$. Then $M\!:=f^{-1}(V)$ is a compact neighborhood of the point  $(x,\sigma)$ and this one is not mapped into $\lambda$\,. Since $\X_\infty\!\times\!{\Si}^\bu$ is compact, finitely many of such $M$ can cover it, forming a compact neighborhood of $\X_\infty\!\times\!{\Si}^\bu$ which is not mapped into $\lambda$\,, thus the second equivalence below follows: 
\begin{eqnarray*}
\lambda\in f^\bu_\infty\big(\X_\infty\!\times \!{\Si}^\bu\big)=\widetilde{f}_\infty\big(\X_\infty\!\times \!{\Si}^\bu\big)    &\Leftrightarrow&\lambda= \widetilde{f}_\infty(x,\si)\,,\hspace{.2cm}\text{ for some }(x,\si)\in\X_\infty\!\times \!\Si^\bu\\
 &\Leftrightarrow&\lambda\in \widetilde{f}_\infty\big(\X_\infty\!\times \!N\big)\,,\hspace{.2cm}\forall\,N\in \mathcal{N}\\
 &\Leftrightarrow&\lambda\in \bigcap_{N\in \mathcal{N}}\widetilde{f}_\infty\big(\X_\infty\!\times \!N\big)\,,
\end{eqnarray*}
hence we finish the proof of the Corollary.
\end{proof}

\begin{Remark}\label{frustiuk} 
{\rm Note that whenever $\X$ is not compact, the point $\lambda=0$ is included in the spectrum.}
\end{Remark}

The next Corollary is a version of Gohberg's Lemma for our case; note however that it consists of equalities, while traditionally one only gets lower bounds for the norm. 
Given a filter base $\mathcal C$ on a topological space $Z$ and a bounded function $\varphi:Z\to \mathbb R$\,, we consider the following definition of {\it upper limit along $\mathcal C$}\,:
\begin{equation}\label{whatelse}
\begin{aligned}
\limsup_{\mathcal C}\varphi:&
=\sup\left\{\mu\in \mathbb R\,\,\big\vert \,\,\forall\, C\in \mathcal C\,,\ \exists\,z\in C\ {\rm such\ that}\ \varphi(z)>\mu\right\}\\
&=\max\bigcap_{C\in\mathcal C}\overline{\varphi(C)}\,.
\end{aligned}
\end{equation}
Note that for $\mathcal S\!:=\{K^c\,\vert \, K\text{ relatively compact}\}$ on a locally compact space (referred to as the {\it standard filter}), one obtains
the usual definition of the upper limit.

\smallskip
Apart from the neighborhood base $\mathcal{B}$\,, we need the following family of subsets of $\X\!\times\Xi$\,:
$$
\mathcal B'\!:=\big\{\X\!\times\!A\!\mid\! A\in\mathcal B|_\Xi\big\}\,.
$$

\begin{Corollary}\label{gondolin}
Let again $f\in\mathbb F_{(1)}^{-1}\widetilde{\mathfrak L}$ and let $\mathcal B$ be a base of neighborhoods of $\,\Si^\bu$ in $\widetilde{\Si}$\,. Then 
\begin{equation}\label{belegherer}
{\rm dist}\big({\sf Op}(f),\mathbf B\big)=\!\!\max_{(x,\si)\in\X_\infty\!\times\Si^\bu}\big\vert f^\bu_\infty(x,\si)\big\vert=\limsup_{\mathcal B'}\vert f\vert\,.
\end{equation}
\end{Corollary}

\begin{proof}
We recall from the proof of Theorem \ref{dagorlat} that the $C^*$-algebra ${\mathbf B}^\bu$ is isomorphic to $\widetilde{\mathfrak B}/\mathfrak B$ and to ${\rm C}_0\big(\X\!\times\!\Si^\bu\big)$, thus it is commutative. The elements of these algebras being normal, their norm coincide with their spectral radius.
Thus one has 
\begin{equation*}\label{beleghier}
{\rm dist}\big({\sf Op}(f),\mathbf B\big)=\big\Vert\,{\sf Op}^\bu(f)\,\big\Vert_{{\mathbf B}^\bu}\!=\rho_{{\mathbf B}^\bu}\big({\sf Op}^\bu(f)\big)\,.
\end{equation*}
Hence, from Theorem \ref{dagorlat}:
\begin{equation*}\label{babolat}
\rho_{{\mathbf B}^\bu}\big({\sf Op}^\bu(f)\big)=
\!\!\max_{(x,\si)\in\X_\infty\!\times\Si^\bu}\big\vert f^\bu_\infty(x,\si)\big\vert\,,
\end{equation*}
and the first equality in \eqref{belegherer} is checked. 

\smallskip
The second one follows straightforwardly from \eqref{golumm} and the last form of the definition \eqref{whatelse}.
\end{proof}

\begin{Example}\label{sepavar}
{\rm Let us suppose that $f(x,\xi)=\gamma(x)\psi(\xi)$\,, which is often written in the form $f=\gamma\otimes\psi$\,. We assume that $\gamma\in\mathcal F^{-1}{\rm L}^1(\Xi)\subset \mathrm C_0(X)$ and $\psi\in\widetilde \J$\,. Then ${\sf Op}(f)={\sf M}_\gamma{\sf C}_{\widehat\psi}$ is the composition between a multiplication operator and the operator of convolution with the Fourier transform of $\psi$ and its spectrum can be very complicated. However, applying Theorem \ref{dagorlat}, the spectrum of its image in the quotient $\mathbf B^\bu$ is just the product in $\mathbb C$ of the subsets $\overline{\gamma(\X)}$ and $\psi^\bu\big(\Si^\bu\big)$\,, while the r.h.s. of \eqref{belegherer} becomes $\max|\gamma|\limsup_{\mathcal B}|\psi|$\,.}
\end{Example}

Let us finally consider the next situation, leading to an extension of our results: Besides the ideal $\J$ and all the objects constructed starting with it, let $\mathrm I$ be another invariant ideal of $\widetilde{\J}$, with Gelfand spectrum $\Si_\mathrm I$\,, such that $\mathrm C_0\subset \J\subset\mathrm I\subset \widetilde{\J}$\,. We are not interested in the $C^*$-algebra $\widetilde{\mathrm I}$ or the quotient $\mathrm I^\bu\!:=\widetilde{\mathrm I}/\mathrm I$\,. One shows easily the isomorphism
$$
\mathrm I^\circ\!:=\widetilde{\mathrm J}/\mathrm I\cong\big(\,\widetilde{\mathrm J}/\mathrm J\big)/\big(\mathrm I/\mathrm J\big)=\mathrm J^\bu/\big(\mathrm I/\mathrm J\big)\,.
$$

At the level of spectra, one has
$$
\Xi\subset\Si\subset\Si_\mathrm I\subset\widetilde\Si\,.
$$

The set difference $\widetilde\Si\!\setminus\!\Si_\mathrm I=:\Si^\circ\!\subset\Si^\bu$ may be seen as the Gelfand spectrum of $\mathrm I^\circ$; it is a compact invariant subset of $\widetilde\Si$\,, formed of fixed points. We also have the ideal $\mathfrak B_{\mathrm I}\lll\widetilde{\mathfrak B}$\,, obtained by applying the inverse Fourier transformation in the first variable to the crossed product $\Xi\!\rtimes\!\mathrm I$\,. The ${\sf Op}$-calculus sends it isomorphically to an ideal $\mathbf B_{\mathrm I}$ of $\widetilde{\mathbf B}$\,. We have 
$$
\mathbf K\big[\mathrm{L}^2(\X)\big]\subset\mathbf B\subset\mathbf B_{\mathrm I}\subset\widetilde{\mathbf B}\subset\mathbf B\big[\mathrm{L}^2(\X)\big]\,.
$$

The image of ${\sf Op}(f)$ into the quotient $\mathbf B^\circ\!:=\widetilde{\mathbf B}/\mathbf B_{\mathrm I}$ will be denoted by ${\sf Op}^\circ(f)$\,.

\begin{Theorem}\label{dragalit}
Let $f\in\mathbb F_{(1)}^{-1}\,\widetilde{\mathfrak L}\subset \mathrm C\big(\X\!\times\!\widetilde{\Si}\big)$\,. Then (with obvious notations)
\begin{eqnarray*}\label{gollumus}
{\rm sp}_{{\mathbf B}^\circ}\big({\sf Op}^\circ(f)\big)&=&f^\circ_\infty\big(\X_\infty\!\times \!{\Si}^\circ\big)\,.
\end{eqnarray*}
If $\,\mathcal D$ is a base of neighborhoods of $\,\Si^\circ$ in $\,\widetilde\Si$ and $\D'\!:=\{X\!\times\!(D\cap\Xi)\!\mid\!D\in\mathcal D\}$\,, then
\begin{equation}\label{biligherer}
{\rm dist}\big({\sf Op}(f),\mathbf B_\mathrm I\big)=\!\!\max_{(x,\si)\in\X_\infty\!\times\Si^\circ}\big\vert f^\circ_\infty(x,\si)\big\vert=\limsup_{\mathcal D'}\vert f\vert\,.
\end{equation}
\end{Theorem}

The proof consists in making minor obvious modifications in the proofs above. The main fact is that the elements of $\Si^\circ$ are still fixed points of the action $\th$.

\begin{Remark}\label{tenebros}
{\rm All the examples of such ideals $\mathrm I$ are provided by the closed subsets $\Si^\circ$ of $\Si^\bu$. Note that the result applies to the same (large) class of symbols. Equality \eqref{biligherer}, compared to \eqref{belegherer}, deals with the distance of the same operator ${\sf Op}(f)$ to a closer ideal, in terms of a smaller number.}
\end{Remark}

%----------------------------------------------------------------------------------
\section{The standard case}\label{stindarde}
%----------------------------------------------------------------------------------

We give now more details concerning the ideal $\mathrm C_0(\Xi)$\,, which is the one usually considered.
For this smallest admissible $C^*$-subalgebra ${\rm C}_0(\Xi)=:\J_{\rm st}$\, (which is also an ideal of $\mathrm{BC_u}$), one refers to $\widetilde \J_{\rm st}\equiv{\rm VO}(\Xi)$ as to the $C^*$-algebra of {\it vanishing oscillation} (or {\it slowly oscillating}) functions on $\Xi$\,. Note that in this standard case one has simply $\Si_{\rm st}^\bu\!=\widetilde\Si_{\rm st}\!\setminus\!\Xi$\,. Set $\widetilde{\mathfrak L}_{\rm st}={\rm L^1}(\Xi;\widetilde{\J}_{\rm st})$\,; at the level of crossed products it follows that ${\sf Sch}:\widetilde{\mathfrak{J}}_{\rm st}\to \mathbf{B}[{\rm L}^2(\Xi)]$ maps the ideal $\mathfrak J_{\rm st}=\Xi\!\ltimes \!{\rm C}_0$ onto the ideal $\mathbf{K}[{\rm L}^2(\Xi)]$ of compact operators. From this fact and diagram (\ref{bigorneau}), it is immediate that ${\sf Op}(\mathfrak B_{\rm st})=:\mathbf B_{\rm st}=\mathbf{K}(\mathscr{H})$\,. As $\widetilde{\mathbf B}_{\rm st}\ll \mathbf{B}(\mathscr{H})$\,, we conclude that $\mathbf{B}^\bu_{\rm st}$ is a $C^*$-subalgebra of $\mathbf B(\mathscr{H})/\mathbf{K}(\mathscr{H})$\,, the {\it Calkin algebra}. In this case the spectrum ${\rm sp}_{\mathbf B_{\rm st}^\bu}\equiv{\rm sp}_{\rm ess}$ is {\it the essential spectrum}, having well known equivalent descriptions.

\smallskip
So one can apply the results of the preceding section in this more explicit setting. But it is possible in this case to make the $G$-lemma more concrete, in terms of a usual asymptotic behavior (but replacing the equality with a lower bound).
We start with a lemma, recalling that $\S$ already denoted the family of all the complements of relatively compact subsets of $\Xi$\,. It is a filter on $\Xi$\,; in the proof we will use some elementary facts about filters and ultrafilters.

\begin{Lemma}
For every $\psi\in \widetilde{\J}_{\rm st}$ one has
$$
\bigcap_{S\in \mathcal{S}}\overline{\psi(S)}\subset \widetilde{\psi}\big(\Si_{\rm st}^\bu\big)\,.
$$
\end{Lemma}

\begin{proof}
 For $\lambda\in \bigcap_{S\in \mathcal{S}}\overline{\psi(S)}$ (the cluster set of $\psi$ under the filter $\S$), there is an ultrafilter $\mathcal{U}$ finer than $\mathcal{S}$ such that $\lim_{\mathcal{U}}\psi=\lambda$\,. Set $j:\Xi\to \widetilde{\Si}_{\rm st}$ for the canonical injection. Then, $j(\mathcal{U})$ is a filter base on $\widetilde{\Si}_{\rm st}$ and $\lim_{j(\mathcal{U})}\widetilde{\psi}=\lambda$\,. Thus, for every ultrafilter $\mathcal{U}'$ on $\widetilde{\Si}_{\rm st}$ finer than the filter defined by $j(\mathcal{U})$\,, one has $\lim_{\mathcal{U}'}\widetilde{\psi}=\lambda$\,. Moreover, as $\widetilde{\Si}_{\rm st}$ is compact, every ultrafilter on it is principal, i.e. it corresponds to the filter of neighborhoods of some $\si\in \widetilde{\Si}_{\rm st}$\,. As $\mathcal{S}$ is a free filter, $\mathcal{U}$ does not have a limit on $\Xi$, hence the accumulation points of $j(\mathcal{U})$ do not intersect $\Xi$\,. This implies that $\si$, the limit of $\mathcal{U}'$, does not belong to $\Xi$\,, and as in the standard case one has $\Si_{\rm st}=\Xi$\,, then $\si\in \Si^\bu_{\rm st}$\,. Hence $\lambda=\widetilde{\psi}(\si)\in \widetilde{\psi}\big(\Si_{\rm st}^\bu\big)$\,.
\end{proof}

For simplicity of notations, we stated and proved this lemma for functions $\psi$ only defined on $\Xi$\,. The same arguments, involving the family $\S'\!:=\{\X_\infty\}\!\times\!\S$ (a filter in the compact space $\X_\infty\!\times\!\widetilde{\Si}$\,)\,, lead to the formula
$$
\bigcap_{S\in \S}\overline{f(X\!\times\!S)}\subset f^\bu_\infty\big(\X_\infty\!\times\!\Si_{\rm st}^\bu\big)\,,\quad\forall\,f\in\mathrm C_0\big(\X\!\times\!\widetilde\Si\big)\supset\mathbb F_{(1)}^{-1}\,\widetilde{\mathfrak L}\,,
$$
making an obvious connection with Theorem \ref{dagorlat}. Instead of \eqref{golumm} one gets
\begin{equation*}\label{phisic}
{\rm sp}_{\rm ess}\big({\sf Op}(f)\big)\supset\bigcap_{S\in \mathcal S}\overline{f\big(\X\!\times\!S\big)}
\end{equation*}
and a $G$-lemma replacing Corollary \ref{gondolin} is
\begin{equation}\label{balegar}
{\rm dist}\big({\sf Op}(f),\mathbf K(\mathscr H)\big)\ge\limsup_{\mathcal S'}\vert f\vert=\min_{x\in \X}\limsup_{\xi \to \infty}|f(x,\xi)|\,.
\end{equation}

The last two formulas, although just inequalities, present the advantage that they are expressed in terms of the behavior of the symbol on the familiar complements of compact sets. Actually the $\limsup$ is also a standard concept here, as seen in \eqref{balegar}.

\smallskip
It is a known fact that if $\mathfrak{A}_1\ll\mathfrak{A}_2$ are $C^*$-algebras and $a\in \mathfrak{A}_1$ has an inverse $b\in \mathfrak{A}_2$\,, then $b\in \mathfrak{A}_1$\,. On the other hand, Atkinson's Theorem \cite[Proposition 3.3.11]{Pedersen} states that an element in $\mathbf{B}(\mathscr{H})$ is {\it a Fredholm operator} if and only if its projection into the Calkin algebra is invertible. In other words, an element ${\sf Op}(f)\in \widetilde{\mathbf{B}}_{\rm st}$ is a Fredholm operator if and only if ${\sf Op}^\bu(f)$ is invertible in $\mathbf{B}^\bu_{\rm st}$. As Theorem \ref{dagorlat} gives an explicit form of the spectrum of ${\sf Op}^\bu(f)$\,, we obtain a criterion to determine when ${\sf Op}(f)$ is a Fredholm operator. As mentioned in Remark \ref{frustiuk}, if $\X$ is not compact, then for every symbol $f$ we have $\Tilde{f}_\infty\!\left(\{x_\infty\}\!\times\! \widetilde{\Si}\right)=\{0\}$\,, thus $0\in {\rm sp}_{\rm ess}\big({\sf Op}(f)\big)$\,. We only state a weak form of a conclusion, using the familiar asymptotic quantity:

\begin{Corollary}\label{siasta}
If $\,\X$ is not compact, for every  symbol $f\in\mathbb F_{(1)}^{-1}\widetilde{\mathfrak L}_{\rm st}$\,, the operator ${\sf Op}(f)$ is not Fredholm. On the other hand, if $\,\X$ is compact, then ${\sf Op}(f)$ is a Fredholm operator if 
$$
\min_{x\in \X}\liminf_{\xi \to \infty}|f(x,\xi)|>0\,.
$$
\end{Corollary}

A sharper but less explicit form (involving equivalence) follows from Corollary \ref{gondolin} applied to this case.

\begin{Remark}\label{fomentez}
{\rm In his work \cite{Gr} V. Grushin also proved a $G$-lemma, only involving a lower bound, for pseudodifferential operators with bounded symbols for $\X=\mathbb R^n$. The symbols used by him are some kind of Hörmander class functions: $f\in {\rm C}^\infty\big(\R^n\!\times \!\R^n\big)$ such that for every $\alpha,\beta,\gamma\in \mathbb N^n$, with $\beta\neq 0$\,, there is a constant $c_{\alpha\gamma}>0$ and a function $h_{\alpha\beta\gamma}\in {\rm C}_0\big(\R^n\big)$ such that 
$$
\left|\frac{\partial ^\alpha}{\partial \xi^\alpha}\frac{\partial^{\beta+\gamma}}{\partial x^{\beta+\gamma}}f(x,\xi)\right|\leq h_{\alpha\beta\gamma}(x)(1+|\xi|)^{-|\alpha|}\,,\hspace{.4cm}\left|\frac{\partial ^\alpha}{\partial \xi^\alpha}\frac{\partial^{\gamma}}{\partial x^{\gamma}}f(x,\xi)\right|\leq c_{\alpha\gamma}(1+|\xi|)^{-|\alpha|}\,.
$$
This presents various differences with respect to our symbols. They require much more stringent regularity. However, due to the condition $\beta\neq 0$\,, the first estimation implies that a Grushin symbol needs no $\textrm C_0$-bahavior in the $x$-variable, as in our case.
On the other hand, the second estimation implies that a Grushin symbol $f$ verifies $f(x,\cdot)\in {\rm VO}(\Xi)$ for every fixed $x$, but they are much more exigent regarding the decaying ratios, involving an infinity of conditions. We do not need smoothness, and when it is present, only one first-order condition in the variable $\xi$ is enough; see also Remark 7.1.}
\end{Remark}

%--------------------------------------------------------------------------------
\section{Further examples and remarks}\label{adunate}
%---------------------------------------------------------------------------------

\begin{Remark}\label{oftat}
{\rm Consider the  case $\Xi=\R^n$, corresponding to $\X=\R^n$. There is a simple device, to be used below, allowing to construct elements of $\widetilde\J$ from the behavior required in $\J$\,: If a bounded function $\psi:\R^n\to\mathbb C$ is differentiable and each first order derivative $\partial _j\psi$ belongs to $\J$ for some admissible $C^*$-subalgebra on $\R^n$, then $\psi\in \widetilde{\J}$\,. This follows from the gradient Theorem for line integrals:
$$
\psi(\xi+\zeta)-\psi(\xi)=\int_{0}^{1}\!\nabla \psi(\xi+t\zeta)\cdot \zeta \,dt\,.
$$}
\end{Remark}

We finish this study by visiting some non-standard examples illustrating the constructions. We focus on the couple $\big(\J,\widetilde{\J}\,\big)$\,, which then generates the entire formalism.
Let $\big\{\J_i\,\big\vert\,i\in I\big\}$ be a family of admissible (non-unital) $C^*$-subalgebras. It is trivial to check that $\bigcap_{i}\!\J_i$ is once again admissible and $\widetilde{\bigcap_{i}\!\J_i}=\bigcap_{i}\!\widetilde\J_i$\,. Such intersections can be used to model complex behavior of symbols. All the following examples are not just $C^*$-subalgebras of $\mathrm{BC_u}(\mathbb{R}^n)$, but also ideals, most of which can be regarded as the ideal of bounded uniformly continuous functions which vanish with respect to some filter.

\begin{Example}\label{stoskan}
{\rm This is just a warm up, since it can be covered in various ways from subsequent examples. For $\Xi=\X=\mathbb R$\,, let us consider the ideal
$$
\J_+\!=\Big\{\phi\in {\rm BC}_{\rm u}(\R)\,\Big\vert\,\lim_{\xi\to +\infty}\phi(\xi)=0\Big\}\,,
$$
which is clearly admissible and much larger than ${\rm C}_0=\J_{\rm st}$\,.
Thus, if the derivative $\psi'$ decays to zero at $+\infty$ and is bounded and uniformly continuous (a very generous condition, to be satisfied ``to the left"), the function $\psi$ belongs to $\widetilde{\J}$\,. For instance, its restriction to an interval $(a,+\infty)$\,, with $a$ large, could be
$$
\psi(\xi)=\sin\big[\beta(\xi)\big]\,,\quad{\rm with}\quad|\beta'(\xi)|\underset{+\infty }{\to}0\,.
$$
Examples as $\beta(\xi):=\xi^\alpha$ where $\alpha<1$ are available (as many others), and this is very far from requiring convergence to some limit. This and the very weak conditions towards $-\infty$ prove that the vanishing oscillation functions with respect to certain filters may form quite large classes of functions.}
\end{Example} 

\begin{Example}\label{rradial}
{\rm If $\o_0\in \mathbb S^{n-1}\subset \R^n=\Xi$\,,  it is easy to see that
$$
\J_{\o_0}=\Big\{\psi\in {\rm BC}_{\rm u}\big(\R^n\big)\,\Big\vert\,\lim_{\xi\to\o^\infty_0}\psi(\xi)=0\Big\}
$$
is an admissible ideal. The meaning of $\lim_{\xi\to\o^\infty_0}\psi(\xi)=0$ is that for every $\epsilon>0$ there exist $V\in\mathcal V(\o_0)$ and $R>0$ such that $|\psi(\xi)|<\epsilon$ if $\xi/|\xi|\in V$ and $|\xi|>R$\,. This can also be interpreted differently if one places $\o_0$ at infinity (on the boundary of a radial compactification, for instance). Note that the previous example corresponds to $n=1$ and $\o_0=1$\,.

\smallskip
Consider for instance $\Xi=\mathbb{R}^3$, $\omega_0=(1,0,0)$ and a function $\psi:\mathbb{R}^3\to \mathbb C$ such that on the space $\xi_1\geq 1$ it is given by $\psi\big(\xi_1,\xi_2,\xi_3\big)=\cos\big(\xi_1^{-2}\cos(\xi_2+\xi_3)\big)$\,. Then 
$$
\nabla \psi\big(\xi_1,\xi_2,\xi_3\big)=\xi_1^{-2}\sin\big(\xi_1^{-2}\cos(\xi_2+\xi_3)\big)\sin(\xi_2+\xi_3)\big(-2\xi_1^{-3} ,1,1\big)\,,
$$ 
which clearly verifies $\partial_{\xi_1}\psi,\partial_{\xi_2}\psi,\partial_{\xi_3}\psi\in \J_{\o_0}$\,, thus $\psi\in \widetilde{\J}_{\o_0}$ by Remark \ref{oftat}.}
\end{Example}

\begin{Example}\label{pesado}
{\rm To make it more general, one could start with a regular dynamical compactification $\O$ of $\Xi$ (that is $\Xi$ is open and dense in $\O$ and the action by translations extend to $\O$). Choose $\O_0$\,, a closed invariant subset of the boundary $\O\!\setminus\!\Xi$\, (an orbit closure, for instance). We define $\mathcal U_\O(\O_0)$ to be the family of neighborhoods of $\O_0$ in $\O$\,, set
$$
\J_{\O_0}=\big\{\psi\in {\rm BC}_{\rm u}\big(\Xi\big)\,\big\vert\,\forall\,\epsilon>0\,,\,\exists\,U\in\mathcal U_\O(\O_0)\ {\rm s.\,t.}\ |\psi(\xi)|<\epsilon\ {\rm if}\ \xi\in U\cap\Xi\big\}\,,
$$
and check that is an admissible ideal.  Note that Example \ref{rradial} corresponds to consider $\O$ to be the {\it radial compactification} of $\mathbb R^n$ and $\O_0=\{\o_0^\infty\}$\,.}
\end{Example}

\begin{Example}\label{pescado}
{\rm Let $E$ be a subset of $\Xi$ such that $EK\ne\Xi$ for every compact $K$. This means that $E$ is {\it not syndetic}. The set $E=\cup_{n\in\N}[a_n,b_n]$\,, composed of disjoint closed intervals of $\Xi=\mathbb{R}$\,, is syndetic if and only if the lengths of the gaps $\alpha_n:=a_{n+1}-b_n$ is bounded. The numbers with prime absolute value in $\Xi=\mathbb Z$ are not syndetic. We define an admissible ideal $\J_E$ of bounded uniformly continuous functions on $\Xi$ by the extra condition
$$
\forall\,\epsilon>0\,,\ \exists\,K\subset\Xi\ {\rm compact\ \,s.t.}\ |\psi(\xi)|<\epsilon\,\ {\rm if}\ \xi\notin EK.
$$
In such a situation we write $\lim_{[E]}\psi=0$\,. Setting $\Xi=\R$ and $E=(-\infty,a]$ for some real $a$ leads to Example \ref{stoskan}, while any non-void interval $[a,b]$ serves to recover the standard case.}

\smallskip
{\rm Let again $\Xi=\R^n$. Chose $E$ to be any non compact submanifold of dimension $m<n$\,. Such $E$ fulfills the previous condition. In this case, we can give generic examples of functions belonging to the ideal $\J_E$\,.  For $\xi\in E$, let $N_\xi$ be the normal space to $E$ at the point $\xi$\,. Pick then some orthonormal base $\{\eta_{\xi,1},\eta_{\xi,2},\dots,\eta_{\xi,n-m}\}$\,. Then $\varphi\in \J_E$ follows from the condition
\begin{equation}\label{cafe}
    \lim_{s\to \pm\infty}\left\{\sup_{\xi\in E}\max_{1\leq i\leq n-m}\varphi\big(\xi+s\eta_{\xi,i}\big)\right\}=0\,.
\end{equation}
Moreover, this implies that if $\psi$ is differentiable and 
$$
\lim_{s\to \pm\infty}\max_{1\leq j\leq n}\left\{\sup_{\xi\in E}\max_{1\leq i\leq n-m}\partial_j\psi(\xi+s\eta_{\xi,i})\right\}=0\,,
$$
then $\psi\in \widetilde{\J}_E$\,. Set for instance $E:=\{\xi\in \mathbb{R}^2: \xi=(t,t^2),\, t\in \mathbb{R}\}$ the graph of the function $\mathbb{R}\ni t\mapsto t^2$.  At any point $(t,t^2)$ the normal space is generated by $(-2t,1)$. Then, one may rewrite condition (\ref{cafe}) as
$$
\lim_{s\to \pm\infty}\sup_{t\in \mathbb{R}}\varphi\left(\left(t-\frac{2ts}{\sqrt{1+4t^2}},t^2+\frac{s}{\sqrt{{1+4t^2}}} \right)\right)=0\,,
$$
which means that for every $\epsilon>0$ there is $\hat{s}>0$ such that outside the set 
$$
\left\{\left(t-\frac{2ts}{\sqrt{1+4t^2}},t^2+\frac{s}{\sqrt{{1+4t^2}}} \right)\,\Big\vert\,t\in \mathbb{R},|s|<\hat{s}\right\}
$$
the function $\varphi$ is $\epsilon$-small.}
\end{Example}

\begin{Example}\label{ceialcezarului}
{\rm Let $\mathbb D:=\big\{D_n\,\big\vert\,n\in\N\big\}$ a family of compact subsets of $\Xi$ satisfying by definition for every $n$\,: (a) $0\ne\wm(D_n)$\,, (b) $D_n\subset D_{n+1}$ and (c) every compact set $L\subset\Xi$ is contained in some $D_n$\,. It follows that $\cup_nD_n=\Xi$ (the group $\Xi$ is $\si$-compact) and $\lim_n\wm(D_n)=\infty$\,. We say that $A\subset\Xi$ is {\it $\mathbb D$-full}, and we write $A\in\mathcal B_\mathbb D$ if it is closed and $\underset{n\to\infty}{\lim}\!\frac{\wm(A\cap D_n)}{\wm(D_n)}=1$\,.
It is useful to note that $\mathcal B_\mathbb D$ is stable under finite intersections, translations by elements $\zeta\in\Xi$ and that it contains all the complements of relatively compact subsets. In addition, if $A\in\mathcal B_\mathbb D$ and $K$ is relatively compact, then $A\!\setminus\!K\in\mathcal B_\mathbb D$\,. 

\smallskip
One defines 
\begin{equation}\label{babel}
\J_\mathbb D:=\big\{\psi\in {\rm BC_u}\,\big\vert\,\forall\,\epsilon>0\,,\,\exists\,A\in\mathcal B_\mathbb D\,\ {\rm s.t.}\ \,|\psi(\xi)|\le\epsilon\ \,{\rm if}\ \,\xi\in A\big\}\,.
\end{equation}
By using the stated properties of the family $\mathcal B_\mathbb D$ one easily shows that $\J_\mathbb D$ is an admissible ideal of ${\rm BC_u}$\,. Slightly formally, the extra condition reads $\psi|_B\in{\rm C}_0(B)$ for some $B\in\mathcal B_\mathbb D$ depending on $\psi$\,. Anyhow, it is easy to give examples: just choose a $\mathbb D$-full subset $A$\,, impose the $\mathrm C_0$-condition on $A$ and let $\psi$ be arbitrary outside $A$\,, up to boundedness and uniform continuity. Then, in $\R^n$, one may use Remark \ref{oftat} to indicate non-trivial examples of elements of the $C^*$-algebra $\widetilde{\J}_\mathbb D$\,. In $\Xi=\R$ one can use balls $D_n\!:=\{\xi\!\mid\!|\xi|\le n\}$\,, and then $\J_\mathbb D$ is connected with the concept of {\it Ces\'aro limit}; see \cite[Lemma 4.1]{GI}: their elements are characterized by the condition
$$
\lim_n\frac{\int_{D_n}\!|\psi|}{\wm(D_n)}=0\,.
$$
The reader may try various examples of infinite unions of closed finite intervals as potential elements of $\mathcal B_\mathbb D$\,.
Belonging to $\J_\mathbb D$ is decided in terms of the lengths and the gaps involved.}
\end{Example}

\bigskip
{\bf Acknowledgements.} We are very thankful to S. Richard for his comments and suggestions, which allowed for an improvement of the results we present here.

\bigskip
{\bf Financial support.}
N. J. has been supported by ANID, Beca de Doctorado Nacional 21220105. M. M. has been supported by the Fondecyt Project 1200884. 

%-------------------------------------------------------------------------------------------------------

\bigskip
ADDRESS

\smallskip
N. Jara:

Facultad de Ciencias, Departamento de Matem\'aticas, Universidad de Chile 

Las Palmeras 3425, Casilla 653, Santiago, Chile.

E-mail: {\it nestor.jara@ug.uchile.cl}

%\smallskip
%M. M\u antoiu:

%Facultad de Ciencias, Departamento de Matem\'aticas, Universidad de Chile 

%Las Palmeras 3425, Casilla 653, Santiago, Chile.

%E-mail: {\it mantoiu@uchile.cl}


\begin{thebibliography}{00}
%-------------------------------------------------------------------------------------------------------

\bibitem{DR} A. Dasgupta and M. Ruzhansky: \emph{Gohberg Lemma, Compactness, and Essential Spectrum of Operators on Compact Lie Groups}, J. d'Analyse Math. {\bf 128}(1), 179--190, (2016).

\bibitem{GI} V. Georgescu and A. Iftimovici: \emph{Riesz–Kolmogorov Compactness Criterion, Lorentz Convergence and Ruelle Theorem on Locally Compact Abelian Groups}, Potential Analysis, {\bf 20}, 265--284, (2004).

\bibitem{Go} I. Gohberg, \emph{On the Theory of Multidimensional Singular Integral Equations}, Soviet Math. Dokl. {\bf 1}, 960--963, (1960).

\bibitem{Gr} V. Grushin: \emph{Pseudo-differential Operators in $\R^n$ with Bounded Symbols}, Funct. Anal. Appl. {\bf 4}, 202--212, (1970).

\bibitem{Ho} L. H\"ormander: \emph{Pseudodifferential Operators and Hypoelliptic Operators}, in Pseudodifferential Operators [Russian translation], Mir, Moskow, 9--62, (1967).

\bibitem{KN} J. J. Kohn and L. Nirenberg: \emph{An Algebra of Pseudodifferential Operators}, in Pseudodifferential Operators [Russian translation], Mir, Moskow, 63--87, (1967).

\bibitem{M} M. M\u antoiu: {\it Compactifications, Dynamical Systems at Infinity and the Essential Spectrum of Generalized Schr\"odinger Operators}, J. reine angew. Math. {\bf 550}, 211--229, (2002).

\bibitem{Ma} M. M\u antoiu: {\it Abelian $C^*$-Algebras which Are Independent with Respect to a Filter}, J. London Math. Soc. {\bf 71}, 740--758, (2005).

\bibitem{Man} M. M\u antoiu: {\it Anisotropic Gohberg Lemmas for Pseudodifferential Operators on Abelian Compact Groups}, Preprint arXiv:2210.02568 (2022).

\bibitem{MW} S. Molahajloo and M. W. Wong. \emph{Ellipticity, Fredholmness and Spectral Invariance of Pseudo-differential Operators on $S^1$}, J. Pseudo-Differ. Oper. Appl., {\bf 1}(2), 183--205, (2010).

\bibitem{Ro} J. Roe: \emph{Lectures on Coarse Geometry}, University Lecture Series 31. Providence, Rhode Island: American Mathematical Society, 2003.

\bibitem{Pedersen} G. K. Pedersen: \emph{Analysis now}. Graduate Texts in Mathematics, 118. Springer-Verlag, New York, 1989.

\bibitem{RT1} M. Ruzhansky and V. Turunen: {\it Pseudodifferential Operators and Symmetries}, Pseudo-Differential Operators: Theory and Applications {\bf 2}, Birkh\"auser Verlag, 2010. 

\bibitem{RT2} M. Ruzhansky and V. Turunen:  {\it Global Quantization of Pseudo-differential Operators on Compact Lie Groups, SU(2) and 3-Sphere}, Int. Math. Res. Not. IMRN, {\bf 11}, 2439--2496, (2013).

\bibitem{RVR} M. Ruzhansky and J.\,P. Velasquez Rodriguez: \emph{Non-Harmonic Gohberg's Lemma, Gershgoshin Theory and Heat Equation on Manifolds with Boundary}, Math. Nachr. {\bf 294}, 1783--1820, (2021).

\bibitem{VR} J.\,P. Velasquez Rodriguez: \emph{On Some Spectral Properties of Pseudo-differential Operators on $\T$}, J. Fourier Analysis Appl., (2019).

\bibitem{VR1} J. P. Velasquez-Rodriguez: \emph{Spectral Properties of Pseudo-differential Operators over the Compact Group of p-adic Integers and Compact Vilenkin Groups}, Preprint arXiv:1912.11407 (2019).

\bibitem{Se} R.\,T. Seeley: \emph{Integro-differential Operators on Vector Bundles}, Matematika, {\bf 11}(2), 57--97, (1967).

\bibitem{Wi} D. Williams: {\it Crossed Products of $C^*$-Algebras}, Mathematical Surveys and Monographs, {\bf 134}, American Mathematical Society, 2007.

\end{thebibliography}
\end{document}